\documentclass[11pt]{amsart}
\usepackage{amsfonts,amssymb,amscd,amsmath,enumerate,verbatim,calc}

\input xy
\xyoption{all}

\def\NZQ{\mathbb}               

\def\ZZ{{\NZQ Z}}

\def\FF{{\NZQ F}}

\def\frk{\mathfrak}               

\def\pp{{\frk p}}

\def\mm{{\frk m}}

\def\Phi{{\frk N}}
\def\g{\gamma}
\def\G{\Gamma}
\def\a{\alpha}
\def\b{\beta}
\def\s{\sigma}
\def\opn#1#2{\def#1{\operatorname{#2}}} 
\opn\chara{char} \opn\length{\ell} \opn\pd{pd} \opn\rk{rk}
\opn\projdim{proj\,dim} \opn\injdim{inj\,dim} \opn\rank{rank}
\opn\depth{depth} \opn\grade{grade} \opn\height{height}
\opn\embdim{emb\,dim} \opn\codim{codim}

\opn\Tr{Tr} \opn\bigrank{big\,rank}
\opn\superheight{superheight}\opn\lcm{lcm}
\opn\trdeg{tr\,deg}
\opn\reg{reg} \opn\lreg{lreg} \opn\ini{in} \opn\lpd{lpd}
\opn\size{size}\opn{\mult}{mult}
\opn\div{div} \opn\Div{Div} \opn\cl{cl} \opn\Cl{Cl}
\opn\Spec{Spec} \opn\Supp{Supp} \opn\supp{supp} \opn\Sing{Sing}
\opn\Ass{Ass} \opn\Min{Min}
\opn\Ann{Ann} \opn\Rad{Rad} \opn\Soc{Soc}
\opn\Syz{Syz} \opn\Im{Im} \opn\Ker{Ker} \opn\Coker{Coker}
\opn\Am{Am} \opn\Hom{Hom} \opn\Tor{Tor} \opn\Ext{Ext}
\opn\End{End} \opn\Aut{Aut} \opn\id{id}

\opn\nat{nat}
\opn\pff{pf}
\opn\Pf{Pf} \opn\GL{GL} \opn\SL{SL} \opn\mod{mod} \opn\ord{ord}
\opn\Gin{Gin}
\opn\Hilb{Hilb}\opn\adeg{adeg}\opn\std{std}\opn\ip{infpt}
\opn\Pol{Pol}
\opn\sat{sat}
\opn\Var{Var}
\def\ra{{\rightarrow}}
\opn\aff{aff} \opn\con{conv} \opn\relint{relint} \opn\st{st}
\opn\lk{lk} \opn\cn{cn} \opn\core{core} \opn\vol{vol}
\opn\link{link} \opn\star{star}
\opn\gr{gr}

\def\pot#1#2{#1[\kern-0.28ex[#2]\kern-0.28ex]}

\opn\dirlim{\underrightarrow{\lim}}
\opn\inivlim{\underleftarrow{\lim}}
\let\iso=\cong

\let\Dirsum=\bigoplus

\let\to=\rightarrow

\def\Implies{\ifmmode\Longrightarrow \else
        \unskip${}\Longrightarrow{}$\ignorespaces\fi}
\def\implies{\ifmmode\Rightarrow \else
        \unskip${}\Rightarrow{}$\ignorespaces\fi}
\def\iff{\ifmmode\Longleftrightarrow \else
        \unskip${}\Longleftrightarrow{}$\ignorespaces\fi}

\let\:=\colon
\newtheorem{Theorem}{Theorem}[section]
\newtheorem{Lemma}[Theorem]{Lemma}
\newtheorem{Corollary}[Theorem]{Corollary}

\newtheorem{Remark}[Theorem]{Remark}

\newtheorem{Definition}[Theorem]{Definition}

\let\epsilon\varepsilon
\let\phi=\varphi
\let\kappa=\varkappa
\def\qed{\ifhmode\textqed\fi
      \ifmmode\ifinner\quad\qedsymbol\else\dispqed\fi\fi}
\def\textqed{\unskip\nobreak\penalty50
       \hskip2em\hbox{}\nobreak\hfil\qedsymbol
       \parfillskip=0pt \finalhyphendemerits=0}
\def\dispqed{\rlap{\qquad\qedsymbol}}

\opn\dis{dis}
\def\pnt{{\raise0.5mm\hbox{\large\bf.}}}

\opn\Lex{Lex}
\def\homgr{{^{*}\hbox{\rm Hom}}}

\begin{document}

\title{Duality and Tameness}

\author{Marc Chardin, Steven Dale Cutkosky, J\"urgen Herzog  and Hema Srinivasan}
\thanks{The second author was partially supported by NSF}

\address{Marc Chardin, Institut Math\'{e}matique de Jussieu
Universit\'{e} Pierre et Marie Curie,
Boite 247,
4, place Jussieu,
F-75252 PARIS CEDEX 05 }\email{chardin@math.jussieu.fr
}
\address{Dale Cutkosky, Mathematics Department,
202 Mathematical Sciences Bldg,
University of Missouri,
Columbia, MO 65211 USA
}\email{dale@math.missouri.edu}

\address{J\"urgen Herzog, Fachbereich Mathematik und
Informatik, Universit\"at Duisburg-Essen, Campus Essen, 45117
Essen, Germany} \email{juergen.herzog@uni-essen.de}

\address{Hema Srinivasan, Mathematics Department,
202 Mathematical Sciences Bldg,
University of Missouri,
Columbia, MO 65211 USA
}\email{srinivasanh@math.missouri.edu <srinivasan@math.missouri.edu>}

\begin{abstract}
We prove a duality theorem for certain graded algebras and show by various examples different kinds of failure of tameness of local cohomology.
\end{abstract}

\maketitle
\section*{Introduction}
The purpose of this paper is to construct examples of strange behavior of local cohomology. In these constructions we follow a strategy that was already used in \cite{CH} and which relates, via a spectral sequence introduced in \cite{HR},  the local cohomology for the two distinguished bigraded prime ideals in  a standard bigraded algebra.

In the first part we consider algebras with  rather general gradings and deduce a similar spectral sequence in this more general situation. A typical example of such an algebra is the Rees algebra of a graded ideal. The proof for the spectral sequence given here is   simpler than that of the corresponding spectral sequence in \cite{HR}.

In the second part of this paper we construct examples of standard
graded rings $A$, which are algebras over a field $K$,  such that the
function
\begin{equation}\label{eqI1}
j\mapsto \mbox{dim}_K(H^i_{A_+}(A)_{-j})
\end{equation}
is an interesting function for $j\gg0$.
In our examples, this dimension will be finite for all $j$.

Suppose that $A_0$ is a Noetherian local ring, $A=\bigoplus_{j\ge 0}A_j$ is a
standard graded ring and set  $A_+:=\bigoplus_{j>0}A_j$.
Let $M$ be a finitely generated graded $A$-module and $\mathcal {F}:=\tilde M$ be the sheafification of $M$ on $Y=\mbox{Proj}(A)$.
We then have graded $A$-module isomorphisms
$$
H^{i+1}_{A_+}(M)\cong \bigoplus_{n\in{\bf Z}}H^i(Y,{\mathcal F}(n))
$$
for $i\ge 1$, and a similar expression for $i=0$ and $1$.

By Serre vanishing,  $H^i_{A_+}(M)_j=0$ for all $i$ and $j\gg0$.
However, the asymptotic behaviour of $H^i_{A_+}(M)_{-j}$ for  $j\gg0$
is much more mysterious.

In the case when $A_0=K$ is a field, the function (\ref{eqI1})
is in fact a polynomial for large enough $j$. The proof is a consequence of
graded local duality (\cite[13.4.6]{BS} or \cite[3.6.19]{BH})  or follows from Serre duality on a projective variety.

For more general $A_0$, $H_{A_+}(M)_{-j}$ are finitely generated $A_0$
modules, but need not have finite length.

The following problem was proposed by Brodmann and Hellus  \cite{BrHe}.
\vskip .2truein
\noindent {\bf Tameness problem:} Are the local cohomology modules
$H^i_{A_+}(M)$ tame? That is, is it true that either
$$
\{ H^i_{A_+}(M)_j\ne 0,\ \forall j\ll 0\} \mbox{ or }\{ H^i_{A+}(M)_j=0,\ \forall j\ll 0 \} ?
$$
\vskip .2 truein
The problem has a positive solution for $A_0$ of small dimension
(some of the references are Brodmann \cite{Br}, Brodmann and  Hellus \cite{BrHe},  Lim \cite{L}, Rotthaus and Sega \cite{RS}).
\begin{Theorem}[\cite{BrHe}]
If $\dim(A_0)\le 2$, then $M$ is tame.
\end{Theorem}

However, it has recently been shown by two of the authors that tameness can fail if $\mbox{dim}(A_0)=3$.

\begin{Theorem}[\cite{CH}] There are examples with $\dim(A_0)=3$ where
$M$ is not tame.
\end{Theorem}

The statement of this example is reproduced in Theorem \ref{TheoremA0} of this paper. The function (\ref{eqI1}) is periodic for large $j$.
Specifically, the function (\ref{eqI1}) is 2 for large even $j$ and is 0
for large odd $j$.

In Theorem \ref{TheoremA1} we construct an example of failure of tameness
of local cohomology which is not periodic, and is not even a quasi polynomial
(in $-j$) for large $j$. Specifically, we have
for $j> 0$,
$$
\mbox{dim}_K(H^2_{A_+}(A)_{-j})= \left\{\begin{array}{ll} 1&\text{ if $j\equiv 0\mbox{ (mod)
}(p+1)$,}\\
1&\mbox{ if $j=p^t$ for some odd $t\ge 0$,}\\
0&\mbox{ otherwise,}
\end{array}\right.
$$
where  the characteristic of $K$ is $p$. We have $p^t\equiv -1 \mbox{
(mod) } (p+1)$ for all odd $t\ge 0$.

We also give an example (Theorem \ref{TheoremA2}) of failure of tameness where (\ref{eqI1}) is a quasi polynomial  with linear
growth in even degree and is 0 in odd degree.

In Theorem \ref{TheoremA3}, we give a tame example, but we have
$$
\lim_{j\rightarrow\infty}\frac{\mbox{dim}_K(H^2_{A_+}(A)_{-j})}{j^3}=54\sqrt{2},
$$
so (\ref{eqI1}) is far from being a quasi polynomial in $-j$ for large $j$.

While the example of \cite{CH} is for $M=\omega_A$, where $\omega_A$ is the canonical
module of $A$, the examples of the paper are all for $M=A$. This allows us to
easily reinterpret our examples as Rees algebras  in Section \ref{Section4}, and thus we have examples
of Rees algebras over local rings for which the above failure of tameness holds.

In the final section, Section \ref{Section5}, we give an analysis of the
explicit and implicit role of
bigraded duality in the construction of the examples, and some comments
on how it effects the geometry of the constructions.

\section{Duality for polynomial rings in two sets of variables}

Let $K$ be any commutative ring (with unit). In later applications $K$ will be mostly a field. Furthermore let $S=K[x_{1},\ldots ,x_{m},y_{1},\ldots ,y_{n}]$, $P =(x_{1},\ldots ,x_{m})$ and
$Q=(y_{1},\ldots ,y_{n})$.

\medskip
The homology of the \v{C}ech complex ${\mathcal C}_{P}(\_ )$ (resp.\
${\mathcal C}_{Q}(\_ )$)
will be denoted by $H_{P}(\_ )$ (resp.\ $H_{Q}(\_ )$). Notice that for any commutative ring $K$, this homology is the
local cohomology supported in $P$ (resp.\ $Q$), as
$P$ and $Q$ are  generated by a regular sequences.

\medskip
Assume that $S$ is $\Gamma$-graded for some abelian group $\Gamma$, and that $\deg (a )=0$ for $a \in K$.
If $x^sy^p\in R$, $\deg (x^sy^p)=l(s)+l'(p)$ with $l(s):=\sum_{i}s_{i}\deg (x_{i})$ and
$l'(p):=\sum_{j}p_{j}\deg (y_{j})$.

\begin{Definition}{\em  Let $I\subset S$ be a $\Gamma$-graded ideal. The $\Gamma$-grading of $S$ is $I$-{\it sharp} if
$H^{i}_{I}(S)_{\gamma}$ is a finitely generated  $K$-module, for every $i$ and $\gamma \in \Gamma$.}
\end{Definition}

\begin{Lemma}
\label{grading}
The following conditions are equivalent:
\begin{enumerate}
\item[\rm (i)] the $\Gamma$-grading of $S$ is $P$-sharp.

\item[\rm (ii)] the $\Gamma$-grading of $S$ is $Q$-sharp.

\item[\rm (iii)] for all $\gamma \in \Gamma$, $| \{ (\alpha , \beta )\ \:\; \alpha\geq 0, \beta\geq 0,\quad l(\alpha )=\gamma +l'(\beta )\}| <\infty$.
\end{enumerate}
\end{Lemma}

Note that if $K$ is Noetherian, $M$ is a finitely generated $\Gamma$-graded $S$-module,
and the $\Gamma$-grading of $S$ is $I$-sharp, then $H^{i}_{I}(M)_{\gamma}$ is a finite $K$-module, for every $i$ and $\gamma \in \Gamma$. This follows from the converging $\Gamma$-graded spectral sequence $H_{p-q}(H^{p}_{I}(\FF ))\Rightarrow H^{q}_{I}(M)$, where $\FF$ is a $\Gamma$-graded
free $S$-resolution of $M$ with $\FF_i$ finite for every $i$.

\medskip
We will assume from now on that the $\Gamma$-grading of $S$ is $P$-sharp (equivalently $Q$-sharp). Set $\sigma= \deg(x_1\cdots x_my_1\cdots y_n)$, and if $N$ is a $\Gamma$-graded module, then let $N^\vee=\Hom_S(N,S(-\sigma))$ and $N^*=\homgr_K(N,K)$ where the $\Gamma$-grading of $N^*$ is given by $(N^*)_\gamma= \Hom_K(N_{-\gamma}, K)$. More generally, we always denote the graded $K$-dual of a
graded module $N$ (over what graded $K$-algebra soever)  by $N^*$.
Finally we denote by $\phi_{\alpha\beta}$ the map $S(-a)\to S(-b)$ induced by multiplication by $x^\alpha y^\beta$ where $a=\deg x^\alpha$ and $b=-\deg y^\beta$. Then

\begin{Lemma}
\label{dualmap}
$H^m_P(\phi_{\alpha\beta})_\gamma\iso H^n_Q( \phi_{\alpha\beta}^\vee)^*.$
\end{Lemma}

\begin{proof} The free $K$-module $H^{m}_{P}(S)_{\gamma}$ is generated by the elements $x^{-s-1}y^{p}$ with $s,p\geq 0$ and
$-l(s)-l(1)+l'(p)=\gamma$, and $H^{n}_{Q}(S)_{\gamma'}$ is generated by the elements $x^{t}y^{-q-1}$ with $t,q\geq 0$ and $l(t)-l'(q)-l'(1)=\gamma'$.

Let $d_\gamma\: H^{m}_{P}(S)_{\gamma} \to (H^{n}_{Q}(S^\vee)^*)_\gamma =H^{n}_{Q}(S)_{-\gamma-\sigma}$ be the $K$-linear map defined by
\[
d_\gamma(x^{-s-1}y^p)(x^ty^{-q-1})=\left\{ \begin{array}{ll}
       1, & \;\textnormal{if  $s=t$ and $p=q$}, \\
       0, & \;\textnormal{else.}
        \end{array} \right.
\]
Then $d_{\g}$ is an isomorphism (because the $\G$-grading of $R$ is $Q$-sharp) and there is a commutative diagram
$$
\begin{CD}
H^{m}_{P}(S)_{\g -a}&  @> {H^{m}_{P}(\phi_{\a \b})_{\g}}>>  & H^{m}_{P}(S)_{\g -b}\\
@V d_{\g -a} VV & & @V d_{\g -b} VV\\
(H^{n}_{Q}(S)_{-\g +a-\s})^{*}& @> {H^{n}_{Q}((\phi_{\a \b}^{\vee})_{-\g})^{*}} >> &
(H^{n}_{Q}(S)_{-\g +b-\s})^{*}.
\end{CD}
$$
The assertion follows.
\end{proof}

As an immediate consequence we obtain

\begin{Corollary}\label{basicduality}{\em  (a)} Let $f\in S$ be an homogeneous element of degree $a-b$, and
$\phi\:\ S(-a)\to S(-b)$ the graded degree zero map induced by multiplication with $f$. Then
$$
H^{m}_{P}(\phi)\simeq H^{n}_{Q}(\phi^{\vee})^{*}.
$$

{\em (b)} Let $\FF$ be a $\G$-graded complex of finitely generated free $S$-modules. Then
\begin{enumerate}
\item[\em (i)] $H^{i}_{P}(\FF)=0$ for $i\not= m$ and $H^{j}_{Q}(\FF)=0$ for $j\not= n$,

\item[\em (ii)]  $H^{m}_{P}(\FF)\simeq  H^{n}_{Q}((\FF)^{\vee})^{*}$.
\end{enumerate}
\end{Corollary}

As the main result of this section we have

\begin{Theorem}
\label{mainduality}
 Assume that $K$ is Noetherian, the $\G$-grading of $S$ is $P$-sharp
(equivalently $Q$-sharp) and $M$ is a finitely generated $\G$-graded $S$-module. Set $\omega_{S/K}:=S(-\s)$.
Let $\FF$ be a minimal $\G$-graded $S$-resolution of $M$. Then,
\begin{enumerate}
\item[{\em (a)}]  For all $i$, there is a functorial isomorphism
$$
H^{i}_{P}(M)\simeq
H_{m-i}(H^{m}_{P}(\FF)).
$$

\item[{\em (b)}] There is a convergent $\G$-graded spectral sequence,
$$
H^{i}_{Q}(\Ext^{j}_{S}(M,\omega_{S/K}))\Rightarrow
H^{i+j-n}(H^{m}_{P}(\FF)^{*}).
$$
In particular, if $K$ is a field, there  is a convergent $\G$-graded spectral sequence,
$$
H^{i}_{Q}(\Ext^{j}_{S}(M,\omega_{S}))\Rightarrow
H^{\dim S-(i+j)}_{P}(M)^*.
$$
\end{enumerate}
\end{Theorem}

\begin{proof} Claim (a) is an immediate consequence of  Corollary \ref{basicduality} via the $\Gamma$-graded spectral sequence $H_{p-i}(H^{p}_{P}(\FF ))\Rightarrow H^{i}_{P}(M)$. For (b), the two spectral sequences arising from the double
complex ${\mathcal C}_{Q}\FF^{\vee}$ have as second terms
respectively
${'E}_{2}^{ij}=H^{i}_{Q}(\Ext^{j}_{S}(M,\omega_{S/K}))$,
${''E}_{2}^{ij}=0$ for $i\not= n$ and ${''E}_{2}^{nj}=
H^{j}(H^{n}_{Q}(\FF^{\vee}))\simeq
H^{j}(H^{m}_{P}(\FF)^{*})$. If further $K$ is a field,
$H^{j}(H^{m}_{P}(\FF)^{*})\simeq (H_{j}(H^{m}_{P}(\FF)))^{*}
\simeq H^{m-j}_{P}(\FF)^{*}$.
\end{proof}

\begin{Corollary}
\label{gamma}
 Under the hypotheses of the theorem, if $K$ is a field, then for
any $\g\in\G$, there are convergent spectral sequences of finite dimensional $K$-vector spaces
$$
H^{i}_{Q}(\Ext^{j}_{S}(M,\omega_{R}))_{\g}\Rightarrow
H^{\dim S-(i+j)}_{P}(M)_{-\g},
$$
$$
H^{i}_{P}(\Ext^{j}_{S}(M,\omega_{R}))_{\g}\Rightarrow
H^{\dim S-(i+j)}_{Q}(M)_{-\g}.
$$
\end{Corollary}

We now consider the special case that $\Gamma=\ZZ^2$, $S:=K[x_1,\ldots ,x_m,y_1,\ldots ,y_n]$ with
$\deg (x_i)=(1,0)$ and $\deg (y_j)=(d_j,1)$ with $d_j\geq 0$. Set $T:=K[x_1,\ldots ,x_m]$ and let $M$ be a  $\Gamma$-graded $S$-module. We view $M$ as a $\ZZ$-graded module by defining $M_k=\Dirsum_j M_{(j,k)}$. Observe that each $M_k$ itself is a graded $T$-module with $(M_k)_j=M_{(j,k)}$ for all $j$.
We also note that $H^i_P(M)_k\iso H^i_{P_0}(M_k)$,  as can been seen from the definition of local cohomology using the \v{C}ech complex. Here $P_0=(x_1,\ldots,x_m)$ is the graded maximal ideal of $T$.

\begin{Corollary}
\label{iso} With the notation introduced, let $s:= \dim S=m+n$ and $d:=\dim M$. Then

\begin{enumerate}
\item[{\em (a)}] $H^0_P(\Ext^{s-d}_{S}(M,\omega_S))\iso   H_Q^{d}(M)^*$ for any $k$,
\item[{\em (b)}]  there is an exact sequence 
$$
0\ra
H^{1}_{P}(\Ext^{s-d}_{S}(M,\omega_S))\ra H^{d-1}_{Q}(M)^{*}\ra H^{0}_{P}(\Ext^{s-d+1}_{S}(M,\omega_S)).
$$
\item[{\em (c)}]  Let $i\geq 2$. If $\Ext^j_{S}(M,\omega_S)$ is annihilated by a power of $P$ for all $s-d<j<s-d+i$,
then there is an exact sequence
$$
\Ext^{s-d+i-1}_{S}(M,\omega_S)\ra
H^{i}_{P}(\Ext^{s-d}_{S}(M,\omega_S))\ra H^{d-i}_{Q}(M)^{*}\ra H^{0}_{P}(\Ext^{s-d+i}_{S}(M,\omega_S)).
$$
\end{enumerate}

In particular, if  $\Ext^j_{S}(M,\omega_S)$ has finite length for all $s-d<j\leq s-d+i_0$, for some integer $i_0$, then
\[
H^i_{P_0}(\Ext^{s-d}_{S}(M,\omega_S)_k)\iso  (H_Q^{d-i}(M)_{-k})^* \quad \text {for all $i\leq i_0$ and $k\gg 0$.}
\]

Consequently, if $M$ is a generalized Cohen-Macaulay module (i.e.\ $\Ext^{s-i}_{S}(M,\omega_S)$ has finite length for all $i\neq d$), and if we set $N=\Ext^{s-d}_{S}(M,\omega_S)$, then
$$H^i_{P_0}(N_k)\iso (H_Q^{d-i}(M)_{-k})^* \quad \text{for all $i$ and all $k\gg 0$.}$$
\end{Corollary}

\begin{proof} (a), (b) and (c) are direct consequences of Corollary \ref{gamma}.  For the application, notice that
if $\g=(\ell,k)\in \Gamma$ with $k\gg 0$ one has $\Ext^{j}_{S}(M,\omega_S)_\g=0$ for all $s-d<j\leq s-d+i_0$. Therefore, for such $\g$, the desired conclusion follows.
\end{proof}

A typical example to which  this situation applies is  the Rees algebra of a graded ideal $I$ in the standard graded polynomial ring $T=K[x_1,\ldots, x_m]$. Say, $I$  is generated be the homogeneous polynomials  $f_1,\ldots, f_n$ with $\deg f_j=d_j$ for $j=1,\ldots,n$. Then the Rees algebra ${\mathcal R}(I)\subset T[t]$ is generated the elements $f_jt$. If we set $\deg f_jt=(d_j,1)$ for all $j$ and $\deg x_i=(1,0)$ for all $i$, then ${\mathcal R}(I)$ becomes a $\Gamma$-graded $S$-module  via the $K$-algebra homomorphism $S\to {\mathcal R}(I)$ with $x_i\mapsto x_i$ and $y_j\mapsto f_jt$.

According to this definition we have ${\mathcal R}(I)_k=I^k$ for all $k$.

Since $\dim {\mathcal R}(I)=m+1$, the module $\omega_{{\mathcal R}(I)}=\Ext_S^{n-1}({\mathcal R}(I),\omega_S)$ is the canonical module of ${\mathcal R}(I)$ (in the sense of \cite[5.\ Vortrag]{HK}).  Recall that if a ring $R$ is a finite $S$-module of dimension $m+1$, the natural finite map $R\ra  \Hom(\omega_R,\omega_R)\iso \Ext_S^{n-1}(\omega_R,\omega_S)
$ is an isomorphism if
and only if $R$  is $S_2$.  Thus in combination with Corollary \ref{iso} we obtain

\begin{Corollary}
\label{rees}
Let $R:={\mathcal R}(I)$. Suppose that  $R_\pp$ is Cohen-Macaulay for all $\pp\neq (\mm,R_+)$ where $\mm=(x_1,\ldots,x_m)$ and $R_+=\Dirsum_{k>0}I^kt^k$. Then
\[
H_\mm^{i}(I^k)\iso (H_{R_+}^{m+1-i}(\omega_R)_{-k})^* \quad\text{for all $i$ and all $k\gg 0$.}
\]
\end{Corollary}

\begin{proof} Since $\omega_R$ localizes, the conditions imply that $(\omega_R)_\pp$ is Cohen-Macaulay for all $\pp\neq (\mm, R_+)$. Hence the natural into map $R\ra R':=\Ext_S^{n-1}(\omega_R,\omega_S)$ has a cokernel of finite length. In particular, $R'_k=R_k=I^k$ for $k\gg 0$. Thus Corollary \ref{iso} applied to $M=\omega_R$ gives the desired conclusion.
\end{proof}

\begin{Remark}
{\em Let $R:={\mathcal R}(I)$. If the cokernel of $R\ra \Hom(\omega_R,\omega_R)$ is annihilated by a power of $R_+$ (in other words, the
blow-up is $S_2$, as a projective scheme over ${\rm Spec}(T)$), then $R'_k=I^k$ for $k\gg 0$ and therefore one has
an exact sequence
\[
0\ra H_\mm^{0}(T/I^k) \ra (H_{R_+}^{m}(\omega_R)_{-k})^* \ra H_\mm^{0}(\Ext^n_{S}(\omega_R,\omega_S)_k)\ra H_\mm^{1}(T/I^k) \ra (H_{R_+}^{m-1}(\omega_R)_{-k})^*
\]
for such a $k$.}
\end{Remark}

\section{A method of constructing examples}\label{Section2}
Suppose that $R=\bigoplus_{i,j\ge0} R_{ij}$ is a standard bigraded algebra over a ring $K=R_{00}$.
Define $R^i=\bigoplus_{j\ge 0}R_{ij}$ and $R_j=\bigoplus_{i\ge 0}R_{ij}$. Define ideals
$P=\bigoplus_{i>0}R^i$ and $Q=\bigoplus_{j>0}R_j$ in $R$. Suppose that $M=\bigoplus_{ij\in{\bf Z}}M_{ij}$ is
a finitely generated, bigraded $R$-module. Define $M^i=\bigoplus_{j\in{\bf Z}}M_{ij}$ and
$M_j=\bigoplus_{i\in{\bf Z}}M_{ij}$. $M^i$ is a graded $R^0$-module and $M_j$ is a graded $R_0$-module. Let $Q_0=R_{01}R^0$, so that $Q=Q_0R$. Let $P_0=R_{10}R_0$ so that $P=R_{10}R$. We have $K$
module isomorphisms
$$
H^l_Q (M)_{m,n}\cong H^l_{Q_0}(M^m)_n
$$
for $m,n\in{\bf Z}$. Let $\widetilde{M^m}$ be the sheafification of the graded $R^0$-module $M^m$
on $\mbox{Proj}(R^0)$. We have $K$ module isomorphisms
$$
H^l_{Q_0}(M^m)_n\cong H^{l-1}(\mbox{Proj}(R^0),\widetilde{M^m}(n))
$$
for $l\ge 2$ and exact sequences
$$
0\rightarrow H^0_{Q_0}(M^m)_n\rightarrow (R^m)_n=R_{m,n}\rightarrow
H^0(\mbox{Proj}(R^0),\widetilde{M^m}(n))\rightarrow H^1_{Q_0}(M^m)_n\rightarrow 0.
$$
We have similar formulas for the calculation of $H^l_P(M)$.

Now  assume that $X$ is a projective scheme over $K$ and ${\mathcal F}_1$ and ${\mathcal F}_2$ are
very ample line bundles on $X$. Let
$$
R_{m,n}=\Gamma(X,{\mathcal F}_1^{\otimes m}\otimes{\mathcal F}_2^{\otimes n}).
$$
We require that $R=\bigoplus_{m,n\ge 0}R_{m,n}$ be a standard bigraded $K$-algebra. We have
$$
X\cong \mbox{Proj}(R_0)\cong \mbox{Proj}(R^0).
$$
The  sheafification of the graded $R^0$-module $R^m$ on $X$ is $\widetilde{R^m}={\mathcal
F}_1^{\otimes m}$, and  the sheafification of the graded $R_0$-module $R_n$ on $X$ is $\widetilde{R_n}\cong {\mathcal F}_2^{\otimes n}$ (Exercise II.5.9 \cite{Ha}).

For $l\ge 2$ we have bigraded isomorphisms
$$
H^l_Q(R)\cong\bigoplus_{m\ge 0}H^l_{Q_0}(R^m)_n \cong\bigoplus_{m\ge 0,n\in{\bf Z}}H^{l-1}(X,{\mathcal
F}_1^{\otimes m}\otimes{\mathcal F}_2^{\otimes n}).
$$
Viewing $R$ as a graded $R_0$ algebra, we thus   have graded isomorphisms
\begin{equation}\label{eqN15}
H^l_Q(R)_n\cong\bigoplus_{m\ge 0}H^{l-1}(X,{\mathcal F}_1^{\otimes m}\otimes{\mathcal F}_2^{\otimes n}),
\end{equation}
for $l\ge 2$ and $n\in{\bf Z}$.
Let $d=\mbox{dim}(R)=\mbox{dim}(X)+2$.

We now further assume that $K$ is an algebraically closed field and $X$ is a nonsingular $K$
variety. Let
$$
V=\text{\bf P}({\mathcal F}_1\oplus{\mathcal F}_2),
$$
a projective space bundle over $X$ with projection $\pi:V\rightarrow X$. Since ${\mathcal
F}_1\oplus{\mathcal F}_2$ is an ample bundle on $X$, ${\mathcal O}_V(1)$ is ample on $V$.
 Since
$$
R\cong\bigoplus_{t\ge 0}\Gamma(V,{\mathcal O}_V(t))
$$
with
$$
\Gamma(V,{\mathcal O}_V(t))\cong\Gamma(X,S^t({\mathcal F}_1\oplus{\mathcal F}_2))\cong\bigoplus_{i+j=t}R_{ij}
$$
 and $R$ is generated in degree 1 with
respect to this grading, ${\mathcal O}_V(1)$ is very ample on $V$ and $R$ is the homogeneous
coordinate ring of the nonsingular projective variety $V$, so that $R$ is generalized Cohen
Macaulay (all local cohomology modules $H_{R_+}^i(R)$ of $R$ with respect to the maximal bigraded
ideal $R_+$ of $R$ have finite length for $i<d$). We further have that $V$ is projectively normal
by this embedding (Exercise II.5.14 \cite{Ha}) so that $R$ is normal.

\section{Strange behavior of local cohomology}\label{Section3}

In \cite{CH}, we
constructed the following example of failure of tameness of local cohomology.
In the example, $R_0$ has dimension 3, which is the lowest possible for
failure of tameness \cite{Br}.

\begin{Theorem}\label{TheoremA0}
Suppose that $K$ is an algebraically closed field. Then there
exists a normal standard graded $K$-algebra $R_0$
  with
$\dim(R_0)=3$, and a normal standard graded $R_0$-algebra $R$ with $\dim(R)=4$ such
that  for $j\gg 0$,
$$
\dim_K(H^2_{Q}(\omega_R)_{-j})= \left\{\begin{array}{ll}
2&\text{ if $j$ is even,}\\
0&\text{ if $j$ is odd,}\\
\end{array}\right.
$$
where $\omega_R$ is the canonical module of $R$, $Q=\bigoplus_{n>0}R_n$.
\end{Theorem}

We first show that the above theorem is also true for the local cohomology of $R$.

\begin{Theorem}\label{TheoremN19}
Suppose that $K$ is an algebraically closed field. Then there
exists a normal standard graded $K$-algebra $R_0$
  with
$\text{dim}(R_0)=3$, and a normal standard graded $R_0$-algebra $R$ with $\text{dim}(R)=4$ such
that  for $j> 0$,
$$
\dim_K(H^2_{Q}(R)_{-j})= \left\{\begin{array}{ll}
2&\text{ if $j$ is even,}\\
0&\text{ if $j$ is odd,}\\
\end{array}\right.
$$
where  $Q=\bigoplus_{n>0}R_n$.
\end{Theorem}

\begin{proof}
We  compute this directly for the $R$ of Theorem \ref{TheoremA0}  from (\ref{eqN15}) and the calculations of \cite{CH}.
Translating from the notation of this paper to the notation of \cite{CH}, we have $X=S$ is an Abelian surface,
${\mathcal F}_1={\mathcal O}_S(r_2laH)$ and ${\mathcal F}_2={\mathcal O}_S(r_2(D+alH))$.

By (\ref{eqN15}) of this paper, for $n\in{\bf N}$, we have
$$
\begin{array}{lll}
\mbox{dim}_K(H^2_Q(R)_n)&=&\sum_{m\ge 0}h^1(X,{\mathcal F}_1^{\otimes m}\otimes{\mathcal F}_2^{\otimes n})\\
&=& \sum_{m\ge 0}h^1(S,{\mathcal O}_S((m+n)r_2alH+nr_2D)).
\end{array}
$$
Formula (1) of \cite{CH} tells us that for $m,n\in{\bf Z}$,
\begin{equation}\label{eqN25}
h^1(S,{\mathcal O}_S(mH+nD))=
\left\{\begin{array}{ll}
2&\mbox{if $m=0$ and $n$ is even,}\\
0&\mbox{otherwise.}
\end{array}\right.
\end{equation}
Thus for $n<0$, we have
$$
\text{dim}_K(H^2_{Q}(R)_{n})= \left\{\begin{array}{ll}
2&\text{ if $n$ is even,}\\
0&\text{ if $n$ is odd,}\\
\end{array}\right.
$$
giving the conclusions of the theorem.
\end{proof}

The following example shows non periodic failure of tameness.

\begin{Theorem}\label{TheoremA1}
Suppose that $p$ is a prime number such that $p\equiv 2\text{ (mod) }3$ and $p\ge 11$. Then there
exists a normal standard graded $K$-algebra $R_0$ over a field $K$ of characteristic $p$ with
$\mbox{dim}(R_0)=4$, and a normal standard graded $R_0$-algebra $R$ with $\mbox{dim}(R)=5$ such
that  for $j> 0$,
$$
\dim_K(H^2_{Q}(R)_{-j})= \left\{\begin{array}{ll} 1&\text{ if $j\equiv 0\mbox{ (mod)
}(p+1),$}\\
1&\mbox{ if $j=p^t$ for some odd $t\ge 0$,}\\
0&\mbox{ otherwise,}
\end{array}\right.
$$
where  $Q=\bigoplus_{n>0}R_n$. We have $p^t\equiv -1
(\mod)  (p+1)$ for all odd $t\ge 0$.
\end{Theorem}

To establish this, we need the following simple lemma.

\begin{Lemma}\label{LemmaA1} Suppose that $C$ is  a non singular curve of genus g over an algebraically closed field $K$,
 and ${\mathcal M}$, $\mathcal N$ are line bundles on $C$. If $\text{deg}({\mathcal M})\ge 2(2g+1)$ and
$\text{deg}({\mathcal N})\ge 2(2g+1)$, then the natural map
$$
\Gamma(C,{\mathcal M})\otimes \Gamma(C,{\mathcal N})\rightarrow \Gamma(C,{\mathcal
M}\otimes{\mathcal N})
$$
is a surjection.
\end{Lemma}

\begin{proof} If ${\mathcal L}$ is a line bundle on $C$, then $H^1(C,{\mathcal L})=0$ if $\text{deg}({\mathcal
L})>2g-2$ and ${\mathcal L}$ is very ample if $\text{deg}({\mathcal L})\ge 2g+1$ (Chapter IV,
Section 3 \cite{Ha}).

Suppose that ${\mathcal L}$ is very ample and ${\mathcal G}$ is another line bundle on $C$. If
$\text{deg}({\mathcal G})>2g-2-\text{deg}({\mathcal L})$, then ${\mathcal G}$ is 2-regular for
${\mathcal L}$ (Lecture 14, \cite{M1}). Thus if $\text{deg}({\mathcal G})>2g-2+\text{deg}({\mathcal
L})$, then
$$
\Gamma(C,{\mathcal G})\otimes\Gamma(C,{\mathcal L})\rightarrow \Gamma(C,{\mathcal
G}\otimes{\mathcal L})
$$
is a surjection by Castelnuovo's Proposition, Lecture 14, page 99 \cite{M1}.

We now apply the above to prove the lemma.  Write ${\mathcal M}\cong {\mathcal A}^{\otimes
q}\otimes {\mathcal B}$ where $\mathcal A$ is a line bundle such that $\text{deg}({\mathcal
A})=2g+1$, and $2g+1\le \text{deg}({\mathcal B})<2(2g+1)$. $\text{deg}({\mathcal
N})>2g-2+\text{deg}({\mathcal A})$. Thus there exists a surjection
$$
\Gamma(C,{\mathcal N})\otimes\Gamma(C,{\mathcal A})\rightarrow \Gamma(C,{\mathcal
A}\otimes{\mathcal N}).
$$
We iterate to get surjections
$$
\Gamma(C,{\mathcal A}^{\otimes i}\otimes{\mathcal N})\otimes
\Gamma(C,{\mathcal A})\rightarrow \Gamma(C,{\mathcal A}^{\otimes(i+1)}\otimes{\mathcal N})
$$
for $i\le q$, and a surjection
$$
\Gamma(C,{\mathcal A}^{\otimes q}\otimes{\mathcal N})\otimes\Gamma(C,{\mathcal B})\rightarrow
\Gamma(C,{\mathcal M}\otimes{\mathcal N}).
$$
\end{proof}

We now prove  Theorem \ref{TheoremA1}.  For the construction, we start with an example from Section 6 of
\cite{CS}. There exists an algebraically closed field $K$ of characteristic $p$, a curve $C$ of genus 2 over  $K$, a point $q\in C$ and a line bundle ${\mathcal M}$ on $C$ of degree 0, such that
for $n\ge 0$,
$$
H^1(C,{\mathcal O}_C(q)\otimes{\mathcal M}^{\otimes n}) =\left\{\begin{array}{ll} 1&\text{ if
}n=p^t\text{ for some }t\ge 0,\\
0&\text{otherwise.} \end{array}\right.
$$
Further, $H^1(C,{\mathcal O}_C(2q)\otimes{\mathcal M}^{\otimes n})=0$ for all $n>0$.

Let $a=p+1$. Let $E$ be an elliptic curve over $K$, and let $T=E\times E$, with projections
$\pi_i:T\rightarrow E$. Let $b\in E$ be a point and  let
${\mathcal A}=\pi_1^*({\mathcal O}_E(b))\otimes\pi_2^*({\mathcal O}_E(b))$. Let $X=T\times C$, with
projections $\phi_1:X\rightarrow T$, $\phi_2:X\rightarrow C$.
Let ${\mathcal L}={\mathcal O}_C(q)$.
Let
$$
\mathcal F_1=\phi_1^*({\mathcal A})^{\otimes a}\otimes\phi_2^*({\mathcal L})^{\otimes a},
$$
$$
\mathcal F_2=\phi_1^*({\mathcal A})^{\otimes (1+a)}\otimes\phi_2^*({\mathcal
L}^{\otimes(1+a)}\otimes{\mathcal M}^{-1}).
$$
For $m,n\ge 0$, we have
\begin{equation}\label{eq3}
\begin{array}{lll}
\Gamma(X,{\mathcal F}_1^{\otimes m}\otimes{\mathcal F}_2^{\otimes n}) &=& \Gamma(T,{\mathcal
A}^{\otimes (ma+n(1+a))})\otimes\Gamma(C,{\mathcal L}^{\otimes (ma+n(1+a))}\otimes{\mathcal
M}^{-\otimes n})\\&=&\Gamma(T,{\mathcal A}^{\otimes a})^{\otimes m}\otimes\Gamma(T,{\mathcal
A}^{\otimes (1+a)})^{\otimes n} \otimes\Gamma(C,{\mathcal L}^{a})^{\otimes
m}\otimes\Gamma(C,{\mathcal
L}^{\otimes(1+a)}\otimes{\mathcal M}^{-1})^{\otimes n}\\
&=&\Gamma(X,{\mathcal F}_1)^{\otimes m}\otimes\Gamma(X,{\mathcal F}_2)^{\otimes n}
\end{array}
\end{equation}
by the K\"unneth formula (IV of Lecture 11 \cite{M1}) and  Lemma \ref{LemmaA1}.

Let $R_{m,n}= \Gamma(X,{\mathcal F}_1^{\otimes m}\otimes{\mathcal F}_2^{\otimes n})$.
$R=\bigoplus_{m,n\ge 0}R_{m,n}$ is a standard bigraded $K$-algebra by (\ref{eq3}). Thus (\ref{eqN15}) holds.

By the Riemann Roch Theorem, we compute,  \begin{equation}\label{eq4} h^0(C,{\mathcal
L}^{\otimes r}\otimes{\mathcal M}^{-\otimes s}) =h^1(C,{\mathcal L}^{\otimes r}\otimes{\mathcal
M}^{-\otimes s})+r-1, \end{equation} and for $s<0$,

\begin{equation}\label{eq5}
h^1(C,{\mathcal L}^{\otimes r}\otimes{\mathcal M}^{-\otimes s}) =\left\{\begin{array}{ll} 1-r& r<0,\\
1&r=0, s< 0,\\
1& r=1,s=-p^t, \mbox{ for some }t\in{\bf N,}\\
0&r=1, s\ne -p^t\mbox{ for some }t\in{\bf N,}\\
0&r=2,s<0,\\
0&r\ge 3.
\end{array}\right.
\end{equation}
 We further have
\begin{equation}\label{eq6}
h^1(T,{\mathcal A}^{\otimes r})=\left\{
\begin{array}{ll}
0&r\ne 0,\\
2&r = 0,
\end{array}\right.
\end{equation}
and
\begin{equation}\label{eq20}
h^0(T,{\mathcal A}^{\otimes r})=\left\{
\begin{array}{ll}
0&r< 0,\\
1&r=0,\\
r^2&r > 0.
\end{array}\right.
\end{equation}

By (\ref{eqN15}), for $n\in{\bf Z}$, we have
$$
\mbox{dim}_K(H^2_Q(R)_{n}) = \sum_{m\ge 0}h^1(X,{\mathcal F}_1^{\otimes m}\otimes
{\mathcal F}_2^{\otimes n}). $$

By the K\"unneth formula,
$$
\begin{array}{lll}
H^1(X,{\mathcal F}_1^{\otimes m}\otimes {\mathcal F}_2^{\otimes n})&\cong &H^0(T,{\mathcal
A}^{\otimes (ma+n(1+a))})\otimes H^1(C,{\mathcal L}^{\otimes (ma+n(1+a))}\otimes{\mathcal
M}^{-\otimes n})\\
&& \oplus H^1(T,{\mathcal A}^{\otimes (ma+n(1+a))})\otimes H^0(C,{\mathcal L}^{\otimes
(ma+n(1+a))}\otimes{\mathcal M}^{-\otimes n}).
\end{array}
$$
Thus by (\ref{eq4}) - (\ref{eq20}), we have for $j> 0$,

$$
\text{dim}_K(H^2_Q(R)_{-j})
 =\left\{
\begin{array}{ll}
1 & j\equiv 0 \mbox{ (mod) }a,\\
1 & j= p^t\mbox{ for some odd } t\in{\bf N,}\\
0&\mbox{otherwise.}
\end{array}\right.
$$
and we have the conclusions of Theorem \ref{TheoremA1}. \vskip .4truein Theorem \ref{TheoremA2} gives
an example of failure of tameness of local cohomology with larger growth.

\begin{Theorem}\label{TheoremA2}
Suppose that $K$ is an algebraically closed field. Then there
exists a normal standard graded $K$-algebra $R_0$ over $K$
  with
$\dim(R_0)=4$, and a normal standard graded $R_0$-algebra $R$ with $\text{dim}(R)=5$ such
that  for $j> 0$,
$$
\dim_K(H^3_{Q}(R)_{-j})= \left\{\begin{array}{ll}
6j&\text{ if $j$ is even,}\\
0&\text{ if $j$ is odd,}\\
\end{array}\right.
$$
where  $Q=\bigoplus_{n>0}R_n$.
\end{Theorem}

\begin{proof} Let $E$ be an elliptic curve over $K$, and let $q\in E$ be a point.
Let ${\mathcal L}={\mathcal O}_E(3q)$. By Proposition IV.4.6 \cite{Ha}, $\mathcal L$ is very ample on $E$, and
\begin{equation}\label{eq11}
\oplus_{n\ge 0}\Gamma(E,{\mathcal L}^{\otimes n})
\end{equation}
is generated in degree 1 as a $K$-algebra.
For $n\in{\bf N}$,
\begin{equation}\label{eq7}
h^0(C,{\mathcal L}^{\otimes n})=\left\{
\begin{array}{ll}
0&n<0,\\
1&n=0,\\
3n&n>0.
\end{array}\right.
\end{equation}
and
\begin{equation}\label{eq8}
h^1(C,{\mathcal L}^{\otimes n})=\left\{
\begin{array}{ll}
-3n&n<0,\\
1&n=0,\\
0&n>0.
\end{array}\right.
\end{equation}

Let $X=E^3$, with the three canonical projections $\pi_i:X\rightarrow E$.
Define
$$
{\mathcal F}_1=\pi_1^*({\mathcal L}^{\otimes 2})\otimes
\pi_2^*({\mathcal L}^{\otimes 2})\otimes
\pi_3^*({\mathcal L}^{\otimes 2})
$$
and
$$
{\mathcal F}_2=\pi_1^*({\mathcal L})\otimes
\pi_2^*({\mathcal L})\otimes
\pi_3^*({\mathcal L}^{\otimes 2}).
$$
Let
$$
R_{m,n}=\Gamma(X,{\mathcal F}_1^{\otimes m}\otimes{\mathcal F}_2^{\otimes n}),
$$
$$
R=\bigoplus_{m,n\ge 0}R_{m,n}.
$$
By (\ref{eq11}) and the K\"unneth formula, $R$ is standard bigraded. By (\ref{eqN15}), the fact that $\omega_X\cong{\mathcal O}_X$  and Serre duality,

$$
\mbox{dim}_K(H^3_Q(R)_{-j})=\sum_{m\ge 0}h^2(X,{\mathcal F}_1^{\otimes m}\otimes
{\mathcal F}_2^{-\otimes j})
=\sum_{m\le 0}h^1(X,{\mathcal F}_1^{\otimes m}\otimes{\mathcal F}_2^{\otimes j})
$$
for $j\in{\bf Z}$.

Now by (\ref{eq7}), (\ref{eq8}) and the K\"unneth formula,
 we have that for $n>0$,
$$
h^1(X,{\mathcal F}_1^{\otimes m}\otimes{\mathcal F}_2^{\otimes n})=\left\{
\begin{array}{ll}
0&\text{ if }2m+n\ne 0,\\
2h^0(X,{\mathcal L}^{\otimes n})&\text{ if }2m+n=0.
\end{array}\right.
$$

Thus
the
conclusions of Theorem \ref{TheoremA2} hold.
\end{proof}

The following theorem gives an example of tame, but still rather strange
local cohomology.
Let $[x]$ be the greatest integer in a real number $x$.

\begin{Theorem}\label{TheoremA3}
Suppose that $K$ is an algebraically closed field. Then there
exists a normal standard graded $K$-algebra $R_0$
  with
$\text{dim}(R_0)=3$, and a normal standard graded $R_0$-algebra $R$ with $\text{dim}(R)=4$ such
that  for $j> 0$,
$$
\dim_K(H^2_{Q}(R)_{-j})=
162\left(j^2\left(\left[\frac{j}{\sqrt{2}}\right]+\frac{1}{2}\right)
-\frac{1}{3}\left[\frac{j}{\sqrt{2}}\right] \left(\left[\frac{j}{\sqrt{2}}\right]+1\right)
\left(2\left[\frac{j}{\sqrt{2}}\right]+1\right)\right)
$$
and
$$
\lim_{j\rightarrow \infty}\frac{\dim_K(H^2_{Q}(R)_{-j})}{j^3} =54\sqrt{2}
$$
where  $Q=\bigoplus_{n>0}R_n$.
\end{Theorem}

\begin{proof}
We use the method of Example 1.6 \cite{Cu}.  Let $E$ be an elliptic curve
over an algebraically closed field $K$,
and let $p\in E$ be a point. Let $X=E\times E$ with projections
 $\pi_i:X\rightarrow E$.
Let $C_1=\pi_1^*(p)$, $C_2=\pi_2^*(p)$ and
$$
\Delta=\{(q,q)\mid q\in E\}
$$
be the diagonal of $X$. We compute (as in \cite{Cu}) that
\begin{equation}\label{eq9}
(C_1^2)=(C_2^2)=(\Delta^2)=0
\end{equation}
and
\begin{equation}\label{eq10}
(\Delta\cdot C_1)=(\Delta\cdot C_2)=(C_1\cdot C_2)=1.
\end{equation}

If ${\mathcal N}$ is an ample line bundle on $X$, then
\begin{equation}\label{eqD7}
H^i(X,{\mathcal N})=0\text{ for }i>0
\end{equation}
by the vanishing theorem of Section 16 \cite{M2}.

Suppose that ${\mathcal L}$ is a very ample line bundle on $X$, and
$\mathcal M$ is a numerically effective (nef) line bundle. Then $\mathcal M$
is 3 regular for $\mathcal L$, so that
$$
\Gamma(X,{\mathcal M}\otimes{\mathcal L}^{\otimes n})\otimes\Gamma(X,{\mathcal L})\rightarrow \Gamma(X,{\mathcal M}\otimes{\mathcal L}^{\otimes (n+1)})
$$
is a surjection if $n\ge 3$.  $C_1+2C_2$ is an ample divisor by
the Moishezon Nakai criterion (Theorem V.1.10 \cite{Ha}), so that
$3(C_1+2C_2)$ is very ample by Lefschetz's theorem (Theorem, Section 17 \cite{M2}). Let
$$
{\mathcal F}_1={\mathcal O}_X(9(C_1+2C_2)).
$$
Then ${\mathcal O}_X$ is 3 regular for ${\mathcal O}_X(3(C_1+2C_2))$, so we
have surjections
$$
\Gamma(X,{\mathcal F}_1^{\otimes n})\otimes \Gamma(X,{\mathcal F}_1)
\rightarrow \Gamma(X,{\mathcal F}_1^{\otimes (n+1)})
$$
 for all $n\ge 1$.

$\Delta+C_2$ is ample by the Moishezon Nakai criterion. Let
$D=3(\Delta+C_2)$. $D$ is very ample by Lefschetz's theorem, and
thus ${\mathcal O}_X(D)\otimes{\mathcal F}_1$ is very ample.
Let
$$
{\mathcal F}_2={\mathcal O}_X(3D)\otimes{\mathcal F}_1^{\otimes 3}.
$$
${\mathcal O}_X$ is 3 regular for ${\mathcal O}_X(D)\otimes {\mathcal F}_1$, so we
have surjections
$$
\Gamma(X,{\mathcal F_2}^{\otimes n})\otimes \Gamma(X,{\mathcal F}_2)
\rightarrow \Gamma(X,{\mathcal F}_2^{\otimes (n+1)})
$$
 for all $n\ge 1$.

for $n>0$ and $m\ge 0$, we have
$$
{\mathcal F}_1^{\otimes m}\otimes{\mathcal F}_2^{\otimes n}
\cong {\mathcal O}_X(3nD)\otimes{\mathcal F}_1^{\otimes(m+3n)}.
$$
Since $D$ is nef, it is 3 regular for ${\mathcal F}_1$, and
we have a surjection for all $m\ge 0$, $n>0$,
$$
\Gamma(X,{\mathcal F}_1^{\otimes m}\otimes{\mathcal F}_2^{\otimes n})\otimes
\Gamma(X,{\mathcal F}_1)\rightarrow \Gamma(X,{\mathcal F}_1^{\otimes (m+1)}\otimes{\mathcal F}_2^{\otimes n}).
$$
Let
$$
R_{m,n}=\Gamma(X,{\mathcal F}_1^{\otimes m}\otimes{\mathcal F}_2^{\otimes n}).
$$
We have shown that $\oplus_{m,n\ge 0}R_{m,n}$ is a standard bigraded $K$-algebra. Thus
(\ref{eqN15}) holds.

For $m,n\in{\bf Z}$, let ${\mathcal G}={\mathcal F}_1^{\otimes m}\otimes{\mathcal F}_2^{\otimes n}$.
As in Example 1.6 \cite{Cu},  and by (\ref{eqD7}) and Serre Duality ($\omega_X\cong{\mathcal O}_X$ since $X$ is an Abelian variety), we deduce that
\begin{enumerate}
\item[1.] $({\mathcal G}^2)>0$ and $({\mathcal G}\cdot {\mathcal F}_1)>0$
imply ${\mathcal G}$ is ample and $h^1(X,{\mathcal G})=h^2(X,{\mathcal G})=0$.
\item[2.] $({\mathcal G}^2)<0$
implies  $h^0(X,{\mathcal G})=h^2(X,{\mathcal G})=0$.
\item[3.] $({\mathcal G}^2)>0$ and $({\mathcal G}\cdot {\mathcal F}_1)<0$
imply ${\mathcal G}^{-1}$ is ample and $h^0(X,{\mathcal G})=h^1(X,{\mathcal G})=0$.
\end{enumerate}

Let $\tau_2=-4-\frac{\sqrt{2}}{2}$ and $\tau_1=-4+\frac{\sqrt{2}}{2}$.

Using (\ref{eq9}) and (\ref{eq10}), we compute
$$({\mathcal F}_1^2)=2\cdot 162, ({\mathcal F}_2)^2=31\cdot 162,
({\mathcal F}_1\cdot{\mathcal F}_2)=8\cdot 162.
$$
We have
$$
\begin{array}{lll}
({\mathcal G}^2)&=&324(m^2+8mn+\frac{31}{2}n^2)\\
&=&324(m-\tau_1n)(m-\tau_2n).
\end{array}
$$
and
$$
({\mathcal G}\cdot{\mathcal F}_1)=324(m+4n).
$$
Since $\tau_2<-4<\tau_1<0$, for $n<0$ and $m\in{\bf Z}$, we have
\begin{enumerate}
\item[1.] $m>\tau_2n$ if and only if ${\mathcal G}^2>0$ and ${\mathcal G}\cdot{\mathcal F}_1>0$
\item[2.] $\tau_1n<m<\tau_2n$ if and only if $({\mathcal G}^2)<0$
\item[3.] $m<\tau_1n$ if and only if $({\mathcal G}^2)>0$ and
$({\mathcal G}\cdot{\mathcal F}_1)<0$.
\end{enumerate}
By the Riemann Roch Theorem for an Abelian surface (Section 16 \cite{M2}),
$$
\chi({\mathcal G})=\frac{1}{2}({\mathcal G}^2).
$$
Thus for $m\in{\bf Z}$ and $n<0$,
$$
h^1(X,{\mathcal G})=
\left\{\begin{array}{ll}
-\frac{1}{2}({\mathcal G}^2)
=-162(m^2+8mn+\frac{31}{2}n^2)
&\mbox{if
 $\tau_1n<m<\tau_2n$,}\\
0&\mbox{otherwise.}\end{array}\right.
$$
For $n\in{\bf Z}$, let $\sigma(n)=\text{dim}_K(H^2_Q(R_n))$.
By (\ref{eqN15}),
$$
\sigma(n)=\sum_{m\ge 0}h^1(X,{\mathcal F}_1^{\otimes m}\otimes{\mathcal F}_2^{\otimes n}).
$$
For $n<0$, we have
$$
\sigma(n)
=-162(\sum_{\tau_1n<m<\tau_2n}(m^2+8mn+\frac{31}{2}n^2)).
$$
Setting $r=m+4n$, we have
$$
\begin{array}{lll}
\sigma(n)&=& -162(\sum_{\frac{\sqrt{2}}{2}n<r<-\frac{\sqrt{2}}{2}n}
(r^2-\frac{1}{2}n^2)\\
&=& -324\sum_{r=1}^{\left[-\frac{n}{\sqrt{2}}\right]}(r^2-\frac{1}{2}n^2)+81n^2\\
&=&-324\left(\frac{1}{6}\left[-\frac{n}{\sqrt{2}}\right]\left(
\left[-\frac{n}{\sqrt{2}}\right]+1\right)
\left(2\left[-\frac{n}{\sqrt{2}}\right]+1\right)
-\frac{1}{2}n^2\left[-\frac{n}{\sqrt{2}}\right]\right)+81n^2\\

&=&162\left(n^2\left(\left[-\frac{n}{\sqrt{2}}\right]+\frac{1}{2}\right)
-\frac{1}{3}\left[-\frac{n}{\sqrt{2}}\right]\left(
\left[-\frac{n}{\sqrt{2}}\right]+1\right)\left(2\left[-\frac{n}{\sqrt{2}}\right]
+1\right)\right).
\end{array}
$$
We thus have the
conclusions of the theorem.
\end{proof}

\section{Strange examples of Rees Algebras}\label{Section4}

Let notation and assumptions be as in Section \ref{Section2}. Since ${\mathcal F}_1$ is ample, there
exists $l>0$ such that $\Gamma(X,{\mathcal F}_1^{\otimes l}\otimes{\mathcal F}_2^{-1})\ne 0$. Thus we have an embedding ${\mathcal F}_2\otimes{\mathcal F}_1^{-l}\subset {\mathcal O}_X$. Let ${\mathcal A}={\mathcal F}_2\otimes{\mathcal F}_1^{-l}$, which we have embedded as an ideal sheaf of $X$. For $j\ge 0$ and $i\ge jl$, let
$$
T_{ij}=\Gamma(X,{\mathcal F}_1^{\otimes i}\otimes {\mathcal A}^{\otimes j})=R_{i-jl,j}.
$$
For $j\ge 0$, let $T_j=\bigoplus_{i\ge jl}T_{ij}$ and $T=\bigoplus_{j\ge 0}T_j$.
Let $B=\bigoplus_{j>0}T_j$. $R\cong T$ as graded rings over $R_0\cong T_0$, although they have different
bigraded structures.
Thus for all $i,j$ we have
\begin{equation}\label{eqD5}
H_B^{i}(T)_{j}\cong H^{i}_Q(R)_{j}.
\end{equation}
$T_1$ is a homogeneous ideal of $T_0$, and
$T$ is the Rees algebra of $T_1$.  Thus all of the examples of  Section \ref{Section3} can be interpreted as Rees algebras over normal rings $T_0$ with isolated singularities.

 We thus obtain the following theorems from Theorems
\ref{TheoremN19} - \ref{TheoremA3}. Theorems  \ref{TheoremD3}, \ref{TheoremN20} and \ref{TheoremD5}
give examples of Rees algebras with non tame local cohomology.

\begin{Theorem}\label{TheoremD3}
Suppose that $K$ is an algebraically closed field.
Then there exists a normal, standard graded
$K$ algebra $T_0$  with $\dim(T_0)=3$ and a graded ideal
$A\subset T_0$ such that the Rees algebra $T=T_0[At]$ of $A$ is normal,
and
for $j> 0$,
$$
\dim_K(H^2_{B}(T)_{-j})= \left\{\begin{array}{ll}
2&\text{ if $j$ is even,}\\
0&\text{ if $j$ is odd.}\\
\end{array}\right.
$$
where $B$ is the graded ideal $AtT$ of $T$.
\end{Theorem}

\begin{Theorem}\label{TheoremN20}
Suppose that $p$ is a prime number such that $p\equiv 2(\mod) 3$ and $p\ge 11$. Then there
exists a normal standard graded $K$-algebra $T_0$ over a field $K$ of characteristic $p$ with
$\dim(T_0)=4$, and a
graded ideal
$A\subset T_0$ such that the Rees algebra $T=T_0[At]$ of $A$ is normal,
and
for $j> 0$,

$$
\dim_K(H^2_{Q}(T)_{-j})= \left\{\begin{array}{ll} 1&\text{ if $j\equiv 0(\mod)
(p+1)$,}\\
1&\mbox{ if $j=p^t$ for some odd $t\ge 0$,}\\
0&\mbox{ otherwise,}
\end{array}\right.
$$
where  $B$ is the graded ideal $AtT$ of $T$. We have $p^t\equiv -1
(\mod)(p+1)$ for all odd $t\ge 0$.
\end{Theorem}

\begin{Theorem}\label{TheoremD5}
Suppose that $K$ is an algebraically closed field.
Then there exists a normal, standard graded
$K$-algebra $T_0$  with $\dim(T_0)=4$ and a graded ideal
$A\subset T_0$ such that the Rees algebra $T=R_0[At]$ of $A$ is normal,
and
for $j> 0$,
$$
\dim_K(H^3_{B}(T)_{-j})= \left\{\begin{array}{ll}
6j&\text{ if $j$ is even,}\\
0&\text{ if $j$ is odd,}\\
\end{array}\right.
$$
where $B$ is the graded ideal $AtT$ of $T$.
\end{Theorem}

\begin{Theorem}\label{TheoremD4}
Suppose that $K$ is an algebraically closed field. Then there
exists a normal standard graded $K$ algebra $T_0$
  with
$\dim(T_0)=3$, and a graded ideal $A\subset T_0$ such that
the Rees algebra $T=T_0[At]$ of $A$ is normal, and
  for $j> 0$,
$$
\dim_K(H^2_{B}(T)_{-j})=
162\left(j^2\left(\left[\frac{j}{\sqrt{2}}\right]+\frac{1}{2}\right)
-\frac{1}{3}\left[\frac{j}{\sqrt{2}}\right]
\left(\left[\frac{j}{\sqrt{2}}\right]+1\right)
\left(2\left[\frac{j}{\sqrt{2}}\right]+1\right)\right)
$$
and
$$
\lim_{j\rightarrow \infty}\frac{\dim_K(H^2_{B}(T)_{-j})}{j^3}
=54\sqrt{2}
$$
where $B$ is the graded ideal $AtT$ of $T$.
\end{Theorem}

By localizing at the graded maximal ideal of $T_0$, we obtain examples
of Rees algebras of local rings with strange local cohomology.

In all  of these  examples, $T_0$ is generalized Cohen Macaulay,
but is not Cohen Macaulay. This follows since in all of these examples,
$$
H^2_{P_0}(R_0)_0\cong H^1(X,{\mathcal O}_X)\ne 0.
$$

\section{Local duality in the examples}\label{Section5}

The example of \cite{CH}, giving failure of tameness of local cohomology,
is stated in Theorem \ref{TheoremA0} of this paper.
The proof of \cite{CH} uses the bigraded local duality theorem
of \cite{HR}, which now follows from the much more general bigraded local
duality theorem, Theorem \ref{mainduality}  and Corollary \ref{iso} of this paper, to conclude that
in our situation, where $R$ is generalized Cohen Macaulay,
\begin{equation}\label{eqN26}
(H_Q^{d-i}(\omega_R)_{-j})^*\cong H_P^i(R)_j
\end{equation}
for $j\gg 0$.

In \cite{CH}, the formula
\begin{equation}\label{eqN27}
\begin{array}{lll}
H_P^i(R)_j&\cong& H^i_{P_0}(R_j)\\
&\cong&\bigoplus_{m\in{\bf Z}}H^{i-1}(X,\widetilde{R_j}(m))\\
&\cong&\bigoplus_{m\in{\bf Z}}H^{i-1}(X,{\mathcal F}_1^{\otimes m}\otimes{\mathcal F}_2^{\otimes j})
\end{array}
\end{equation}
for $i\ge 2$ and $j\ge 0$ is then used with formula (1) of \cite{CH} ((\ref{eqN25}) of this paper)
to prove Theorem \ref{TheoremA0}.

In Section \ref{Section2} we derive (\ref{eqN15}) from which we directly compute
the local cohomology in the examples of this paper. We make essential use of Serre duality
on $X$ in computing the examples.

In this section, we show how (\ref{eqN26}) can be obtained directly from
the geometry of $X$ and $V$, and how this formula can be directly interpreted
as Serre duality on $X$.

Let notation be as in Section \ref{Section2}, so that $K$ is an algebraically closed field,
$\mathcal{F}_1$ and $\mathcal{F}_2$ are very ample line bundles on the nonsingular variety $X$.

Let $\omega_R$ be the dualizing module of $R$, and let $\omega_X$ be the canonical bundle of $X$
(which is a dualizing sheaf on $X$).  For a $K$ module $W$, let
$W'=\mbox{Hom}_K(W,K)$.

\begin{Lemma}
We have that
$$
(\omega_R)_{ij}=\left\{
\begin{array}{ll}
\Gamma(X,{\mathcal F}_1^{\otimes i}\otimes{\mathcal F}_2^{\otimes j}\otimes\omega_X)&\text{if $i\ge 1$ and $j\ge 1$}\\
0&\text{otherwise.}
\end{array}\right.
$$
Set $(\omega_R)^i=\bigoplus_{j\in{\bf Z}}(\omega_R)_{i,j}$, a graded $R^0$ module. The
sheafification of $(\omega_R)^i$ on $X$  is
\begin{equation}\label{eqN12}
\widetilde{(\omega_R)^i}= \left\{\begin{array}{ll} {\mathcal F}_1^{\otimes i}\otimes\omega_X&\mbox{
if $i\ge 1$}\\
0&\mbox{ if $i\le 0$.}
\end{array}\right.
\end{equation}
Set $(\omega_R)_j=\bigoplus_{i\in{\bf Z}}(\omega_R)_{i,j}$, a graded $R_0$ module. The
sheafification of $(\omega_R)_j$ on $X$  is
\begin{equation}\label{eqN13}
\widetilde{(\omega_R)_j}= \left\{\begin{array}{ll} {\mathcal F}_2^{\otimes j}\otimes\omega_X&\mbox{
if $j\ge 1$}\\
0&\mbox{ if $j\le 0$.}
\end{array}\right.
\end{equation}
\end{Lemma}

\begin{proof}
Give $R$ the grading where the elements of degree $e$ in $R$ are  $[R]_e=\sum_{i+j=e}R_{ij}$.

We have realized $R$ (with this grading) as the coordinate ring of the projective embedding of
$V={\bf P}({\mathcal F}_1\oplus{\mathcal F}_2)$ by the very ample divisor ${\mathcal O}_V(1)$, with
projection $\pi:V\rightarrow X$.

Let $\omega_V$ be the canonical line bundle on $V$.
We first calculate
 $\omega_V$. Let $f$ be a fiber of the map $\pi:V\rightarrow X$.
By adjunction, we have that $(f\cdot\omega_V)=-2$. Since
$$
\text{Pic}(V)\cong {\bf Z}{\mathcal O}_V(1)\oplus \pi^*(\text{Pic}(X)),
$$
we see that there exists a line bundle ${\mathcal G}$ on $X$ such that
$$
\omega_V\cong {\mathcal O}_V(-2)\otimes\pi^*({\mathcal G}).
$$
The natural split exact sequence
\begin{equation}\label{eqD1}
0\rightarrow {\mathcal F}_2\rightarrow {\mathcal F}_1\oplus{\mathcal F}_2\rightarrow {\mathcal
F}_1\rightarrow 0
\end{equation}
determines a section $X_0$ of $X$, such that $\pi_*$ of the exact sequence
$$
0\rightarrow {\mathcal O}_V(1)\otimes{\mathcal O}_V(-X_0)\rightarrow {\mathcal O}_V(1)\rightarrow
{\mathcal O}_V(1)\otimes{\mathcal O}_{X_0} \rightarrow 0
$$
is (\ref{eqD1}) (Proposition II.7.12 \cite{Ha}). Thus
$$
{\mathcal O}_V(1)\otimes{\mathcal O}_V(-X_0)\cong\pi^*({\mathcal F}_2)
$$
and
$$
{\mathcal O}_V(1)\otimes{\mathcal O}_{X_0}\cong {\mathcal F}_1.
$$
By adjunction, we have that the canonical line bundle of $X_0$ is
$$
\omega_{X_0}\cong \omega_V\otimes{\mathcal O}_V(X_0)\otimes{\mathcal O}_{X_0}.
$$
 Putting the above
together, we see that
$$
{\mathcal G}\cong {\mathcal F}_1\otimes{\mathcal F}_2\otimes\omega_{X}.
$$
Thus
$$
{\omega_V}\cong{\mathcal O}_V(-2)\otimes\pi^*({\mathcal F}_1\otimes{\mathcal
F}_2\otimes\omega_{X}).
$$
We realize $R$ as a bigraded quotient of a bigraded polynomial ring
$$
S=K[x_1,\ldots,x_m,y_1,\ldots,y_n],
$$
 with $\text{deg}(x_i)=(1,0)$ for all $i$ and
$\text{deg}(y_j)=(0,1)$ for all $j$. Viewing $S$ as a graded $K$-algebra with the grading
determined by $d(x_i)=d(y_j)=1$ for all $i,j$, we have a projective embedding $V\subset {\bf
P}=\text{Proj}(S)$. Since $V$ is nonsingular, we see from Section III.7 of \cite{Ha} that
$$
\omega_V\cong{\mathcal Ext}_{\bf P}^r({\mathcal O}_V,{\mathcal O}_{\bf p}(-e))
$$
where $e=m+n$ is the dimension of $S$, and $r=e-\text{dim}(R)$. $\omega_R$ is defined as
$$
\omega_R=\text{*Ext}^r_S(R,S(-e)) \cong \bigoplus_{m\in{\bf Z}}\text{Ext}^r_{\bf P}({\mathcal
O}_V,{\mathcal O}_{\bf P}(m-e)).
$$
For $m\gg0$,
$$
\Gamma({\bf P},{\mathcal Ext}_{\bf P}^r({\mathcal O}_V,{\mathcal O}_{\bf p}(m-e))) \cong
\text{Ext}^r_{\bf P}({\mathcal O}_V,{\mathcal O}_{\bf P}(m-e))
$$
(by Proposition III.6.9 \cite{Ha}). Thus $\omega_R$ and
$$
\Gamma_*(\omega_V) =\bigoplus_{m\in{\bf Z}}\Gamma(V,\omega_V(m))
$$
are isomorphic in high degree. Since  both modules have depth $\ge 2$ at the maximal bigraded
ideal of $R$, we see that
$$
\omega_R\cong \Gamma_*(\omega_V).
$$

Thus
$$
\begin{array}{lll}
\omega_R
&=&\bigoplus_{m\in{\bf Z}}\Gamma(V,\omega_V(m))\\
&=&\bigoplus_{m\in{\bf Z}}\Gamma(V,{\mathcal O}_V(m-2)\otimes\pi^*({\mathcal F}_1\otimes{\mathcal
F}_2\otimes\omega_X)).
\end{array}
$$
Since a fiber $f$ of $\pi$ satisfies $(f\cdot {\mathcal O}_V(m-2)\otimes\pi^*({\mathcal
F}_1\otimes{\mathcal F}_2))<0$ if $m<2$, we see that (with this grading) $[\omega_R]_m=0$ if $m<2$
and For $m\ge 2$, we have
$$
\begin{array}{lll}
[\omega_R]_m&=& \Gamma(X,S^{m-2}({\mathcal F}_1\oplus{\mathcal F}_2)\otimes{\mathcal F}_1\otimes{\mathcal F}_2)\\
&=&\bigoplus_{i+j=m-2}\Gamma(X,{\mathcal F}_1^{\otimes(i+1)}\otimes{\mathcal
F}_2^{\otimes(j+1)}\otimes\omega_X).
\end{array}
$$
The conclusions of the lemma now follow.
\end{proof}

Suppose that $2\le i\le d-2$. Since ${\mathcal F}_1$ and ${\mathcal F}_2$ are ample, and $d-(i+1)>0$, there exists a
natural number $n_0$ such that

\begin{equation}\label{eqN1}
H^{d-(i+1)}(X,{\mathcal F}_1^{\otimes m}\otimes{\mathcal F}_2^n\otimes\omega_X)=0
\end{equation}
for $n\ge n_0$ and all $m\ge 0$.

By (\ref{eqN12}), we have graded isomorphisms
\begin{equation}\label{eqN3}
H^i_Q(\omega_R)_n\cong \bigoplus_{m\ge 1}H^{i-1}(X,{\mathcal F}_1^{\otimes m}\otimes{\mathcal
F}_2^{\otimes n}\otimes\omega_X)
\end{equation}
for $n\in{\bf Z}$.

By Serre duality,
\begin{equation}\label{eqN4}
H^i_Q(\omega_R)_n\cong \bigoplus_{m\ge 1}(H^{d-i-1}(X,{\mathcal F}_1^{-\otimes m}\otimes{\mathcal
F}_2^{-\otimes n}))'.
\end{equation}
By (\ref{eqN1}), there exists $n_0$ such that

\begin{equation}\label{eqN10}
H^i_Q(\omega_R)_{-n}\cong\bigoplus_{m\in{\bf Z}}(H^{d-i-1}(X,{\mathcal F}_1^{-\otimes
m}\otimes{\mathcal F}_2^{\otimes n}))'
\end{equation}
for $n\ge n_0$.

Now apply the functor
$L^*=\mbox{Hom}_K(L,K)$ on  graded $R_0$-modules, with the  grading
$$
(L^*)_{i}=\mbox{Hom}_K(L_{-i},K)
$$
to (\ref{eqN10}), and compare with (\ref{eqN27}), to obtain

\begin{equation}\label{eqN17}
H_P^{d-i}(R)_n\cong (H^i_Q(\omega_R)_{-n})^*
\end{equation}
for $n\ge n_0$, from which (\ref{eqN26}) immediately follows.

We can now verify that Theorem \ref{TheoremA0} is in fact true for all $j>0$,
using (\ref{eqN3}) and (\ref{eqN25}).

We finally comment that an  alternate proof of Theorem \ref{TheoremN19} for $j\gg0$ is obtained from Theorem \ref{TheoremA0},
Formulas (\ref{eqN15}) and (\ref{eqN3}), the fact that $X$ is an Abelian variety
so that $\omega_X\cong{\mathcal O}_X$, and the observation that
$$
h^1(X,{\mathcal F}_2^{-\otimes n})=h^1(X,{\mathcal F}_2^{\otimes n})=0
$$
for $n>0$.

\end{document}